\documentclass[12pt,twosided,reqno]{amsart}
\usepackage{color}

\newtheorem{theorem}{Theorem}[section]
\newtheorem{lemma}[theorem]{Lemma}
\newtheorem{cor}[theorem]{Corollary}
\newtheorem{prop}[theorem]{Proposition}
\usepackage{enumerate,amssymb}
\theoremstyle{definition}
\newtheorem{definition}[theorem]{Definition}
\newtheorem{example}[theorem]{Example}

\theoremstyle{remark}

\numberwithin{equation}{section}

\def\bD{\mathbb{D}}
\def\bM{\mathbb{M}}

\def\bR{\mathbb{R}}
\def\diag{\mathrm{diag}}

\begin{document}
\baselineskip=15pt

\title{Norm and anti-norm inequalities for positive semi-definite matrices}

\author{Jean-Christophe Bourin}
\address{Laboratoire de math\'ematiques, Universit\'e de Franche Comt\'e,
25030 Besan\c con, France}
\email{}

\author{Fumio Hiai}
\address{Graduate School of Information Sciences, Tohoku University,
Aoba-ku, Sendai 980-8579, Japan}



\thanks{}

\begin{abstract}
Some subadditivity results involving symmetric (unitarily invariant) norms are obtained.
For instance, if $g(t)=\sum_{k=0}^m a_kt^k$ is a polynomial of degree $m$ with non-negative
coefficients, then, for all positive operators $A,\,B$ and all symmetric norms,
\begin{equation*}
\| g(A+B) \|^{1/m} \le \|g(A) \|^{1/m} + \| g(B) \|^{1/m}.
\end{equation*} 
To give parallel superadditivity results, we investigate {\it anti-norms}, a class of
functionals containing the Schatten $q$-norms for $q\in(0,1]$ and $q<0$. The results are
extensions of the Minkowski determinantal inequality. A few estimates for block-matrices
are derived. For instance, let $f:[0,\infty)\to [0,\infty)$ be concave and
$p\in(1,\infty)$. If $f^p(t)$ is superadditive, then
${\mathrm{Tr\,}} f(A) \ge \bigl(\sum_{i=1}^m f^p(a_{ii})\bigr)^{1/p}$
for all positive $m\times m$ matrix $A=[a_{ij}]$. Furthermore, for
the normalized trace $\tau$, we consider functions $\varphi(t)$ and $f(t)$ for
which the functional $A\mapsto\varphi\circ\tau\circ f(A)$ is convex or concave, and
obtain a simple analytic criterion.
\end{abstract}

\maketitle

\section{Introduction}

Let $g(t)$ be a convex function on the positive half-line vanishing at $0$. There exist
several matrix versions of the  obvious superadditivity property $g(a+b)\ge g(a)+g(b)$.
It is less usual to mix convexity assumptions and subadditivity results. This paper points
out some matrix subadditivity inequalities involving convex functions such as the first
inequality stated in the abstract. In a parallel way, we also consider superadditivity
inequalities with concave functions, implying some estimate such as the second inequality
of the abstract. These lead us to study  a wide class of functionals on the cone of
positive definite matrices, that we call anti-norms.

The next section  surveys a short list of known subadditivity and superadditivity
inequalities. Section 3 introduces the class of anti-norms and Section 4 contains our new
results. Finally in Section 5, we discuss some simple convexity/concavity criteria for a
functional $A\mapsto\varphi\circ\tau\circ f(A)$, where the range of $f(t)$
is included in the domain of $\varphi(t)$ and $\tau$ is the normalized
trace ($\tau(I) =1$).

By operator, we mean a linear operator on a finite-dimensional Hilbert space. We use
interchangeably the terms operator and matrix. Especially, a positive operator means a
positive (semi-definite) matrix. Consistently $\bM_n$ denotes the set of operators on a
space of dimension $n$ and $\bM_n^+$ stands for its positive part. Though we confine to
$\bM_n$, some extensions to the infinte dimensional setting and operator algebras are
possible.

\section{Concave functions and unitary orbits}

Here we recall some recent subadditivity properties for concave functions, and similarly
superadditivity properties of convex functions. Section 4 will be devoted to our new
inequalities contrasting by involving convex functions and subadditivity properties, or
concave functions and superadditivity results. Given two Hermitian operators, the relation
$X\le Y$ refers to the usual positive semi-definite ordering. From \cite{AB} we know:

\vskip 10pt
\begin{theorem}\label{T-2.1}
Let $f:[0,\infty)\to [0,\infty)$ be concave and $A,\,B$ be positive operators. Then, for
some unitaries $U,\,V$,
$$
f(A+B)\le Uf(A)U^* + Vf(B)V^*.
$$
\end{theorem}

\vskip 5pt
Thus, the obvious scalar inequality $f(a+b)\le f(a)+f(b)$ can be extended to positive 
matrices $A$ and $B$ by considering elements in the unitary orbits of $f(A)$ and of $f(B)$. 
This inequality via unitary orbits considerably improves the famous Rotfel'd trace
inequality for non-negative concave functions and positive operators,
\begin{equation}\label{F-2.1}
{\mathrm{Tr\,}}f(A+B)  \le {\mathrm{Tr\,}}f(A) + {\mathrm{Tr\,}}f(B),
\end{equation}
and its symmetric norm version
\begin{equation}\label{F-2.2}
\| f(A+B)\|  \le \|f(A)\| + \|f(B)\|.
\end{equation}
By definition, a symmetric norm $\|\cdot\|$ on $\bM_n$ satisfies the unitary invariance
condition 
$
\| A\| =\| UA\| =\| AU\|
$
for all  $A$ and all unitaries $U$. 

We can employ Theorem \ref{T-2.1} to get an inequality for positive block-matrices
$$
\begin{bmatrix} A &X \\
X^* &B\end{bmatrix}\in \bM_{n+m}^+, \qquad A\in\bM_n^+, \,  B\in\bM_m^+,
$$
which nicely extend \eqref{F-2.2}. Combined with a useful decomposition for elements in
$\bM_{n+m}^+$ noticed in \cite{B1},
\begin{equation}\label{F-2.3}
\begin{bmatrix} A &X \\
X^* &B\end{bmatrix} = U
\begin{bmatrix} A &0 \\
0 &0\end{bmatrix} U^* +
V\begin{bmatrix} 0 &0 \\
0 &B\end{bmatrix} V^*
\end{equation}
for some unitaries $U,\,V\in \bM_{n+m}$, Theorem \ref{T-2.1} entails a recent theorem of
Lee \cite{Lee}: 

\vskip 10pt
\begin{cor}\label{C-2.2}
Let $f(t)$ be a non-negative concave function on $[0,\infty)$. Then, given an arbitrary
partitioned positive semi-definite matrix,
$$
\left\| \,f\left( \begin{bmatrix} A &X \\ X^* &B\end{bmatrix}\right)
 \right\|
\le \left\| f(A) \right\| +  \left\| f(B) \right\|
$$
for all symmetric norms.
\end{cor}

Applied to $X=A^{1/2}B^{1/2}$, Lee's result yields the Rotfel'd inequalities \eqref{F-2.1}
and \eqref{F-2.2}. In case of the trace norm, the above result may be restated as a trace
inequality without any non-negative assumption: For all concave function $f(t)$ on the
positive half-line and for all positive block-matrices,
$$
{\mathrm{Tr\,}} f\left( \begin{bmatrix} A &X \\ X^* &B\end{bmatrix}\right)
\le {\mathrm{Tr\,}} f(A) +  {\mathrm{Tr\,}} f(B).
$$
The case of $f(t)=\log t$ then gives {\it Fisher's inequality}
$$
\det \begin{bmatrix} A &X \\ X^* &B\end{bmatrix} \le \det A \det B.
$$

Theorem \ref{T-2.1} may be used to extend another classical (superadditivity and concavity)
property of the determinant, {\it Minkowski's inequality}, stating that for
$A,\,B\in \bM_n^+$,
\begin{equation}\label{F-2.4}
{\det}^{1/n} (A+B) \ge {\det}^{1/n} A + {\det}^{1/n} B.
\end{equation}
In fact, Theorem \ref{T-2.1} is  true for any {\it monotone} (i.e., non-decreasing or
non-increasing) and concave function on $[0,\infty)$ such that $f(0)\ge 0$, see
\cite[Theorem 2.1]{AB}. Hence:

\vskip 10pt
\begin{cor}\label{C-2.3}
Let $g:[0,\infty)\to [0,\infty)$ be convex with $g(0)=0$ and let $A,\,B$ be positive
operators. Then, for some unitaries $U,\,V$,
$$
g(A+B)\ge Ug(A)U^* + Vg(B)V^*.
$$
\end{cor}

This convexity version of Theorem \ref{T-2.1} may then be used to refine \eqref{F-2.4} as
\begin{equation}\label{F-2.5}
{\det}^{1/n} g(A+B) \ge {\det}^{1/n} g(A) + {\det}^{1/n} g(B)
\end{equation}
for all $A,\,B\in \bM_n^+$ and all non-negative convex functions $g(t)$ vanishing at $0$.

Minkowski's inequality means that $X\mapsto\det^{1/n}X$ is concave on $\bM_n^+$. 
By using another estimate with unitary orbits, this concavity aspect of \eqref{F-2.4} can
be generalized too. Recall that a linear map $\Phi : \bM_n\to \bM_m$ is called a unital
positive linear map if $\Phi$ preserves positivity and identity. Denote by
$\bM_n\{\Omega\}$ the set of Hermitian operators in $\bM_n$ with spectra in an interval
$\Omega\subset\bR$. Then, we have from \cite{B1}, \cite{B2}:

\vskip 10pt
\begin{theorem}\label{T-2.4}
Let $\Phi : \bM_n\to \bM_m$ be a unital positive linear map, let $f(t)$ be a concave
function on an interval $\Omega$, and let $A,B\in \bM_n\{\Omega\}$. Then, for some
unitaries $U,\,V\in\bM_m$,
\begin{equation*}
f(\Phi(A))\ge \frac{U\Phi( f(A))U^*+V\Phi( f(A))V^*}{2}.
\end{equation*}
If furthermore $f(t)$ is monotone, then we can take $U=V$.
\end{theorem}

This statement for positive maps contains several Jensen type inequalities. The simplest
one is obtained by taking $\Phi : \bM_{2n}\to \bM_n$,
\begin{equation*}
\Phi\left(\begin{bmatrix}  A &X \\ Y &B\end{bmatrix}\right) :=\frac{A+B}{2}.
\end{equation*}
With $X=Y=0$, Theorem \ref{T-2.4} then says: {\it If $A,B\in \bM_n\{\Omega\}$ and  
$f(t)$ is a concave function on  $\Omega$, then, for some unitaries $U,\,V\in\bM_n$,
\begin{equation}\label{F-2.6}
f\left(\frac{A+B}{2}\right)\ge
\frac{1}{2}\left\{U\frac{f(A)+f(B)}{2}U^*+V\frac{f(A)+f(B)}{2}V^*\right\}.
\end{equation}
If furthermore $f(t)$ is monotone, then we can take $U=V$.}

By combining \eqref{F-2.6} and \eqref{F-2.4} we obtain an extension of the concavity aspect
of \eqref{F-2.4}: {\it If $f(t)$ is a non-negative concave function on  $\Omega$ and if
$A,B\in \bM_n\{\Omega\}$, then
\begin{equation}\label{F-2.7}
{\det}^{1/n} f\left(\frac{A+B}{2}\right) \ge
\frac{ {\det}^{1/n} f(A) + {\det}^{1/n} f(B)}{2}.
\end{equation}
}

\vskip 5pt
As another example of combination of Theorem \ref{T-2.1} and \eqref{F-2.3}, we have
\cite{B1}:

\vskip 10pt
\begin{cor}\label{C-2.5}
Let $f:[0,\infty)\to [0,\infty)$ be concave and let $A=[a_{ij}]$ be a positive operator
on a space of dimension $m$. Then, for some rank one ortho-projections $\{E_i\}_{i=1}^m$,
$$
f(A)\le \sum_{i=1}^m f(a_{ii})E_i.
$$
\end{cor}

\vskip 5pt
This refines the standard majorization inequality
${\mathrm{Tr\,}} f(A) \le  \sum_{i=1}^m f(a_{ii})$.
The next sections propose some variations on this relation and other Minkowski type
inequalities. For a detailed background on majorization and unitarily invariant norms,
see for instance \cite{Bh}, \cite{H1}, \cite{S}. A proof of decomposition \eqref{F-2.3}
will be given within the proof of Corollary 4.4 below for the convenience of the reader.

\section{Anti-norms on positive operators}

Symmetric norms on $\bM_n$ can be defined by their restriction to the positive part
$\bM_n^+$. The following axioms are required: (a) $\|A\|=\|UAU^*\|$  for all unitary
$U\in\bM_n$ and all $A\in\bM_n^+$; (b) $\| \lambda A\| =\lambda \| A\|$ for all real
$\lambda\ge0$  and $A\in\bM_n^+$; (c) for all $A,\,B\in \bM_n^+$,
$$
\| A\| + \| B\| \ge \|A+B \| \ge \| A\|.
$$
Indeed, a symmetric norm on $\bM_n$ satisfies (a)--(c), and conversely, if a functional
$\|\cdot\|:\bM_n^+\to[0,\infty)$ satisfies (a)--(c), then $\|X\|:=\|\,|X|\,\|$ for
$X\in\bM_n$ is a symmetric norm on $\bM_n$ as far as $\|\cdot\|$ is not identically zero.
 
\vskip 5pt
The following definition then seems natural:

\begin{definition}\label{D-3.1}
A functional on the cone $\bM_n^+$ taking values in $[0,\infty)$,
$A\mapsto \|A\|_!$, is called a {\it symmetric anti-norm} if the following two conditions
are fulfilled:
\vskip 5pt
\begin{itemize}
\item[1.] It is homogeneous and concave (equivalently, superadditive), that is,
$$
\| \lambda A\|_! =\lambda \| A\|_! \qquad {\mathrm{and}}
\qquad  \| A+B\|_! \ge \| A\|_! + \| B\|_!
$$
for all real $\lambda\ge 0$ and all $A,\, B\in\bM_n^+$.
\vskip 5pt
\item[2.] It is unitarily invariant (or symmetric), that is,
$$
\| A\|_! = \| UAU^*\|_!
$$
for all unitaries $U\in\bM_n$ and all $A\in\bM_n^+$.
\end{itemize}
\vskip 5pt
For a general $X\in \bM_n$, we  may define $\|X\|_! = \|\, |X|\,\|_!$ and then obtain a
symmetric anti-norm on the whole space $\bM_n$. If furthermore $\|X\|_!=0$ implies that
$X=0$, then $\|\cdot\|_!$ is called {\it regular}.
\end{definition}

Similarly to the usual symmetric norms, symmetric anti-norms are defined by symmetric
{\it anti-gauge} functions on $\bR_+^n=[0,\infty)^n$. For $A\in\bM_n^+$ we write
$\lambda(A)=(\lambda_1(A),\dots,\lambda_n(A))$ for the eigenvalue vector of $A$ arranged
in decreasing order with multiplicities.

\begin{prop}\label{P-3.2}
There is a bijective correspondence between the symmetric anti-norms $\|\cdot\|_!$ on
$\bM_n$ and the homogeneous and concave functions $\Phi_!:\bR_+^n\to[0,\infty)$ that are
invariant under coordinate permutations, determined by
$\|A\|_!=\Phi_!(\lambda(A))$ for $A\in\bM_n^+$.
\end{prop}

\begin{proof}
The proof is similar to the usual symmetric norm case (see \cite[4.4.3]{H1} for example),
by using the Ky Fan majorization $\lambda(A+B)\prec\lambda(A)+\lambda(B)$ for
$A,B\in\bM_n^+$. The details are left to the reader.
\end{proof}

\begin{example}\label{E-3.3}
The trace norm is  an anti-norm! More generally for $k=1,\dots,n$, we define the
{\it Ky Fan $k$-anti-norm} on $\bM_n$ as the sum of the $k$ smallest singular values, i.e.,
$$
\|A\|_{\{ k\}} :=\sum_{j=1}^k\mu_{n+1-j}(A),
$$
where $\mu_1(A)\ge\dots\ge\mu_n(A)$ are the singular values of $A$ in decreasing order with
multiplicities. From the min-max principle, it is well-known that this functional is
concave. The anti-norm $\|\cdot\|_{(k)}$ is not regular except for $k=n$ (the trace norm).
\end{example}

\begin{example}\label{E-3.4}
Let $q\in(0,1)$. The {\it Schatten $q$-norms} (or $q$-quasi-norms) on $\bM_n$
$$
\|A\|_{q} :=\left(\sum_{j=1}^n\mu^q_{j}(A)\right)^{1/q}
$$
are symmetric regular anti-norms. The terminology  Schatten $q$-anti-norms allows to
distinguish with the norm case ($q>1$). See Proposition 3.7 below for the proof of a more
general result.
\end{example}

The next two examples are related to the classical geometric and harmonic means. These
anti-norms are not regular.

\begin{example}\label{E-3.5}
The Minkowski functional $A\mapsto{\det}^{1/n} A$ is a symmetric anti-norm on  $\bM_n^+$.
This follows from the previous example by noticing that
${\det}^{1/n}A =\lim_{q\to 0}n^{-1/q}\|A\|_q$.
\end{example}

\begin{example}\label{E-3.6}
The harmonic anti-norm on $\bM_n$ is defined by
$$
\| A\|_{-1}:=\left(\sum_{j=1}^n \mu_j^{-1}(A)\right)^{-1}
$$
if $A$ is invertible, and vanishes on non-invertible operators. More generally
for $r<0$,
$$
\| A\|_{r}:=\left(\sum_{j=1}^n \mu_j^{r}(A)\right)^{1/r}
$$
is a symmetric anti-norm of Schatten type with negative exponent. Concavity of these
functions may be checked by arguing as in the next proof or from  Proposition \ref{P-3.2}
and the fact that the function
$a\in(0,\infty)^n\mapsto\bigl(\sum_{i=1}^na_i^r\bigr)^{1/r}$ is concave for $r<0$.
\end{example}

\begin{prop}\label{P-3.7}
Let $A\mapsto\|A\|_!$ be a symmetric anti-norm on  $\bM^+_n$. Then, so is also
$A\mapsto\|A^q\|_!^{1/q}$ for any $q\in(0,1)$.
\end{prop}

\begin{proof}
The functional $A\mapsto\|A^q\|_!^{1/q}$ is homogeneous; let us check that it is also
superadditive. Let $A,\,B\in\bM_n^+$ and suppose $\|A^q\|_!^{1/q}=\|B^q\|_!^{1/q}=1$.
Since $t\mapsto t^q$ is operator concave, we have, for all $\lambda\in(0,1)$,
$$
(\lambda A + (1-\lambda)B)^q \ge \lambda A^q +(1-\lambda)B^q
$$
so that 
$$
\| (\lambda A + (1-\lambda)B)^q \|_! \ge \|\lambda A^q +(1-\lambda)B^q\|_!
\ge \lambda\| A^q\|_! + (1-\lambda)\| B^q\|_! =1.
$$
Hence,
\begin{equation}\label{F-3.1}
\| (\lambda A + (1-\lambda)B)^q \|^{1/q}_! \ge 1.
\end{equation}
Now, for general $X,Y\in \bM_n^+$ with $\|X\|_!,\,\|Y\|_!>0$, set
$$
A=\frac{X}{\|X^q\|_!^{1/q}},\qquad B=\frac{Y}{\|Y^q\|_!^{1/q}}
$$
and pick
$$
\lambda=\frac{\|X^q\|_!^{1/q}}{\|X^q\|_!^{1/q} + \|Y^q\|_!^{1/q}}.
$$
Then \eqref{F-3.1} yields that 
$$
\|(X+Y)^q\|_!^{1/q}\ge \|X^q\|_!^{1/q} + \|Y^q\|_!^{1/q},
$$
which also holds if $\|X^q\|_!$ or $\| Y^q\|_!$ vanishes.
\end{proof}

\vskip 5pt
\begin{example}\label{E-3.8}
For $k=1,\dots,n$, the functional
$$
\Delta_k(A):=\left(\prod_{j=1}^k \mu_{n+1-j}(A)\right)^{1/k}
$$
is a symmetric anti-norm on $\bM_n$. Indeed, if $A\ge 0$, 
$
\Delta_k(A)=\min \det^{1/k} A_{\mathcal{S}},
$
where the minimum runs over the $k$-dimensional subspaces ${\mathcal{S}}$ and
$A_{\mathcal{S}}$ stands for the compression onto ${\mathcal{S}}$. Hence, given
$A\,,B\ge 0$,
$$
\Delta_k(A+B)=\min {\det}^{1/k} (A+B)_{\mathcal{S}}
= \min {\det}^{1/k} (A_{\mathcal{S}}+B_{\mathcal{S}})
$$
and Minkowski's inequality implies $\Delta_k(A+B)\ge \Delta_k(A) +\Delta_k(B)$.
\end{example}

\vskip 5pt
The geometric and harmonic means, $(a,b)\mapsto\sqrt{ab}$ and
$(a,b)\mapsto2(a^{-1}+b^{-1})^{-1}$, are homogeneous concave functions on the pairs of
positive numbers. Hence:

\vskip 5pt 
\begin{prop}\label{P-3.9}
If $\|\cdot\|_{\ast}$ and  $\|\cdot\|_{\circ}$ are two symmetric anti-norms, then so are
their geometric mean $\sqrt{\|\cdot\|_{\ast}\|\cdot\|_{\circ}}$ and harmonic mean 
$2(\|\cdot\|_{\ast}^{-1}+\|\cdot\|_{\circ}^{-1})^{-1}$.
\end{prop}

\vskip 5pt 
Indeed, more generally, if $m$ is any homogeneous and jointly concave mean on non-negative
numbers (this is the case for operator means \cite{KA}), then
$\|\cdot\|_*\,m\,\|\cdot\|_\circ$ is also a symmetric anti-norm in the above situation.

\vskip 10pt 
\begin{example}\label{E-3.10}
Given an non-decreasing sequence $0\le w_1 \le w_2\le \cdots \le w_n$, the maps
$$
A \mapsto \sum_{k=1}^n w_k\mu_k(A) \qquad {\mathrm{and}} \qquad
A \mapsto \left(\prod_{k=1}^n\mu_k^{w_k}(A) \right)^{\frac{1}{w_1+\cdots +w_n}}
$$
are symmetric anti-norms on $\bM_n$. Indeed, positive sums and weighted geometric means
of anti-norms are still symmetric anti-norms, and Examples \ref{E-3.3} and \ref{E-3.8}
are used.
\end{example}

\vskip 5pt
We now turn to a few consequences of the previous section. Corollary \ref{C-2.3} implies:

\vskip 10pt
\begin{cor}\label{C-3.11}
Let $g:[0,\infty)\to [0,\infty)$ be convex, $g(0)=0$, and let
$A,\,B\in\bM_n^+$. Then, for all symmetric anti-norms,
$$
\|g(A+B) \|_!\ge \| g(A)\|_!+ \|g(B)\|_!.
$$
\end{cor}

\vskip 5pt
In case of the trace norm, this is the convex function version of Rotfel'd inequality
\eqref{F-2.1}. The convexity requirement on $g(t)$ cannot be relaxed to a mere
superadditivity assumption; indeed take for $s,t>0$,
$$
A=\frac{1}{2}
\begin{bmatrix} s&\sqrt{st} \\
\sqrt{st}&t
\end{bmatrix},
\qquad B=\frac{1}{2}
\begin{bmatrix} s&-\sqrt{st} \\
-\sqrt{st}&t
\end{bmatrix},
$$
and observe that the trace inequality $\|g(A+B) \|_1\ge \| g(A)\|_1+ \|g(B)\|_1$ combined
with $g(0)=0$ means that $g(t)$ is convex. For the functional of Example \ref{E-3.5}, we
recapture \eqref{F-2.5} and, with $g(t)=t$, Minkowski's inequality. 
 
Symmetric anti-norms behave well for positive linear maps between matrix spaces. The next
two corollaries follow from Theorem \ref{T-2.4}. Recall that $\bM_n\{\Omega\}$ stands for
the set of Hermitian operators with spectra in an interval $\Omega$. From \eqref{F-2.6} we
have:

\vskip 5pt
\begin{cor}\label{C-3.12}
Let $f:\Omega\to [0,\infty)$ be concave and let $A,\,B\in\bM_n\{\Omega\}$. Then, for all
symmetric anti-norms,
$$
\left\|f\left(\frac{A+B}{2}\right)\right\|_! \ge \left\| \frac{f(A)+(B)}{2}\right\|_!
\ge\frac{\|f(A)\|_!+\|f(B)\|_!}{2}.
$$
\end{cor}

\vskip 5pt
Given a contraction $Z\in \bM_n$, there exists some $K\in \bM_n$ such that $Z^*Z+K^*K=I$.
By using the unital positive map from $\bM_{2n}$ to $\bM_n$
$$
\begin{bmatrix}A &X \\ Y&B \end{bmatrix}\mapsto Z^*AZ +K^*BK,
$$
and letting $X=Y=B=0$, we infer from Theorem \ref{T-2.4}:

\vskip 5pt
\begin{cor}\label{C-3.13}
Let $f:\Omega\to [0,\infty)$ be concave, $0\in\Omega$, let $A\in\bM_n\{\Omega\}$ and let
$Z\in\bM_n$ be a contraction. Then, for all symmetric anti-norms,
\begin{equation}\label{F-3.2}
\| f(Z^*AZ) \|_! \ge \| Z^*f(A)Z\|_!.
\end{equation}
\end{cor}

\vskip 5pt
It is a matrix version of the obvious scalar inequality $f(za)\le zf(a)$ for $z\in[0,1]$,
$a\in\Omega$. In case of the trace norm, it  was noticed by Brown and Kosaki \cite{BK}.
If $f(t)$ is non-negative, concave and {\it monotone} on $\Omega$, Theorem \ref{T-2.4}
shows that \eqref{F-3.2} holds for symmetric norms too. If $f(t)$ is non-negative and
{\it operator} concave on $\Omega$, then $f(Z^*AZ) \ge Z^*f(A)Z$ so that once again
\eqref{F-3.2} holds for symmetric norms too. However, it would be surprising if
\eqref{F-3.2} were true in general for symmetric norms; thus an explicit counterexample
would be desirable.

If $\Omega=[0,\infty)$ and if $Z$ is no longer a contraction, but, in an opposite way, an
expansive operator (i.e., its inverse is a contraction), then one might expect that a
reverse inequality
\begin{equation}\label{F-3.3}
\| f(Z^*AZ) \|_! \le \| Z^*f(A)Z\|_!
\end{equation}
holds. This is not true, as shown by simple counterexamples, for the anti-norms
of Example \ref{E-3.3}, except for the trace norm. In fact the symmetric norm version of
\eqref{F-3.3} holds, see \cite{BL} and references therein. Nevertheless \eqref{F-3.3}
might be true for Schatten $q$-anti-norms, $q\in(0,1)$.

Like symmetric norms, symmetric anti-norms are functions of the singular values as stated
in Proposition \ref{P-3.2}. However, the notion of symmetric anti-norms is more flexible
than that of symmetric norms as it is illustrated by Proposition \ref{P-3.9} and the
sample of previous examples. Thus, one cannot expect for symmetric anti-norms a set of
inequalities as rich as in the symmetric norm case. For instance, \eqref{F-3.3} collapses
for general anti-norms. The next section is successful in giving a few anti-norm estimates.

\section{subadditivity and superadditivity}

The class of convex and subadditive functions $s:[0,\infty)\to [0,\infty)$ is small;
such a function is a sum $s(t)=at+b(t)$ for some  $a\ge0$ and some non-increasing convex
function $b(t)\ge0$. Hence, the inequality for positive  $A,\, B,$
\begin{equation*}
{\mathrm{Tr\,}}s(A+B)  \le {\mathrm{Tr\,}}s(A) + {\mathrm{Tr\,}}s(B)
\end{equation*}
is trivial. Most of convex functions $g:[0,\infty)\to [0,\infty)$ are far from being
subadditive; if $g(0)=0$, they are automatically superadditive. To obtain subadditivity
results, we will assume that composing $g(t)$ with $t\mapsto t^q$ for some $q\in(0,1)$ 
yields a subadditive function. A parallel approach yields superadditivity results for
anti-norms.

\vskip 10pt 
\begin{theorem}\label{T-4.1}
Let $g,f:[0,\infty)\to [0,\infty)$, $0<q<1<p$, and let  $A,\,B\in\bM_n^+$. 
\begin{itemize}
\item[(1)] If $g(t)$ is convex and $g^q(t)$ is subadditive, then for all symmetric norms,
\begin{equation}\label{F-4.1}
\| g(A+B) \|^{q} \le \|g(A) \|^{q} + \| g(B) \|^{q}.
\end{equation} 

\item[(2)] If  $f(t)$ is concave and $f^p(t)$ is superadditive, then for all symmetric
anti-norms,
\begin{equation}\label{F-4.2}
\| f(A+B) \|_!^{p} \ge \|f(A) \|_!^{p} + \| f(B) \|_!^{p}.
\end{equation}
\end{itemize}
\end{theorem}

\vskip 10pt
Note that \eqref{F-4.1} generalizes the fact that $X\mapsto \Vert \,|X|^{1/q} \|^q$ is a
norm on $\bM_n$. Indeed, standard majorizations extend \eqref{F-4.1} to $\bM_n$ as follows:
{\it Let $g:[0,\infty)\to [0,\infty)$ be convex and increasing, and let $q\in(0,1)$. If
$g^q(t)$ is subadditive, then the map $X \mapsto \left\| g(|X|)\right\|^q$ is subadditive
on $\bM_n$.}
Similarly, \eqref{F-4.2} is an extension of Proposition \ref{P-3.7}. 

To prove \eqref{F-4.2} we will need a Ky Fan principle for anti-norms. Recall that the
Ky Fan principle for symmetric norms on $\bM_n^+$ states that $\|A\|\le \| B\|$ for all
symmetric norms if and only if the eigenvalues of $A$ are weakly majorized by those of $B$,
that is, $\|A\|_{(k)}\le\|B\|_{(k)}$ for the Ky Fan $k$-norms, $1\le k\le n$. We express
it by writing $A\prec_w B$. By using the notation $A\prec^w B$ we mean that
$\| A\|_{[k]}\ge \| B\|_{[k]}$ for every anti-norm of Example \ref{E-3.3}.

\begin{lemma}\label{L-4.2}
Let $A,\,B\in\bM_n^+$. If $A\prec^{w}B$, then $\| A\|_!\ge \|B\|_!$ for all symmetric
anti-norms.
\end{lemma}

\vskip 5pt
\begin{proof}
Denote by $\lambda_1(A)\ge\lambda_2(A)\ge\cdots\ge\lambda_n(A)$ the eigenvalues of $A$
arranged in decreasing order. By assumption,
\begin{align*}
\lambda_n(A) &\ge \lambda_n(B), \\
\lambda_n(A) + \lambda_{n-1}(A)  &\ge \lambda_n(B) + \lambda_{n-1}(B),  \\
&\ \,\vdots \\
\lambda_n(A) + \lambda_{n-1}(A)+\cdots +\lambda_1(A)
&\ge \lambda_n(B) + \lambda_{n-1}(B)  +\cdots +\lambda_1(B).  \\
\end{align*}
Hence, replacing $\lambda_1(B)$ with $\lambda_1(B)+r$ for some $r\ge 0$, we may obtain a
majorization
$$
A\prec C
$$
where $C$ is the diagonal matrix
$C=\diag(\lambda_1(B)+r, \lambda_2(B),\cdots,\lambda_n(B))$. We then have some unitaries
$\{U_i\}_{i=1}^m$ and some non-negative scalars $\{\alpha_i\}_{i=1}^m$ of sum $1$ (we may
take $m=n$, see \cite{Zh})  such that
$$
A=\sum_{i=1}^m \alpha_i U_iCU_i^*
$$
so that by concavity and unitary invariance of anti-norms,
$
\| A\|_! \ge \| C \|_! \ge \| B \|_! 
$.
\end{proof}

We turn to the proof of the theorem.

\vskip 5pt
\noindent{\it Proof of Theorem \ref{T-4.1}.\enspace}
Let $A^{\downarrow}$ denote the diagonal matrix listing the eigenvalues of $A$ arranged
in decreasing order.

(1)\enspace From the Ky-Fan majorization
\begin{equation}\label{F-4.3}
A+B\prec A^{\downarrow}+ B^{\downarrow}
\end{equation}
and the convexity of $g(t)$ we infer the weak majorization
$$
g(A+B)\prec_w g(A^{\downarrow}+ B^{\downarrow}),
$$
that is,
$$
g(A+B)\prec_w (g^q(A^{\downarrow}+ B^{\downarrow}))^{1/q}.
$$
By the subadditivity assumption
$g^q(A^{\downarrow}+ B^{\downarrow})\le g^q(A^{\downarrow})+ g^q(B^{\downarrow})$ combined
with the previous weak-majorization, we obtain
$$
g(A+B)\prec_w \{g^q(A^{\downarrow})+ g^q(B^{\downarrow})\}^{1/q}.
$$
Thus
$$
\| g(A+B)\| \le \| \{g^q(A^{\downarrow})+ g^q(B^{\downarrow})\}^{1/q}\|
$$
so that
\begin{align*}
\| g(A+B)\|^q &\le \|   \{g^q(A^{\downarrow})+ g^q(B^{\downarrow})\}^{1/q}\|^q \\
& \le \|g(A) \|^{q} + \| g(B) \|^{q},
\end{align*}
where the last step follows from the well-known fact that $X\mapsto \| X^{1/q}\, \|^q$ is
concave on $\bM_n^+$ (this may also be proved in a similar way to Proposition \ref{P-3.7}).

(2)\enspace From the majorization \eqref{F-4.3} and the concavity of $f(t)$ we infer the
super-majorization 
$$
f(A+B) \prec^w f(A^{\downarrow}+ B^{\downarrow}),
$$
that is,
$$
f(A+B) \prec^w \{f^p(A^{\downarrow}+ B^{\downarrow})\}^{1/p}.
$$
Combine this with the superadditivity assumption
$f^p(A^{\downarrow}+B^{\downarrow})\ge f^p(A^{\downarrow})+ f^p(B^{\downarrow})$ to obtain
$$
f(A+B) \prec^w \{f^p(A^{\downarrow})+ f^p(B^{\downarrow})\}^{1/p}.
$$
Thus, by Lemma \ref{L-4.2},
$$
\| f(A+B) \|_! \ge \|  \{f^p(A^{\downarrow})+ f^p(B^{\downarrow})\}^{1/p} \|_!,
$$
and applying Proposition \ref{P-3.7} yields the result.
\qed

\vskip 5pt
The most important special case of \eqref{F-4.1} is:

\vskip 10pt
\begin{cor}\label{C-4.3}
Let $g(t)=\sum_{k=0}^m a_kt^k$ be a polynomial of degree $m$ with all non-negative
coefficients. Then, for all positive operators $A,\,B$ and all symmetric norms,
\begin{equation*}
\| g(A+B) \|^{1/m} \le \|g(A) \|^{1/m} + \| g(B) \|^{1/m}.
\end{equation*} 
\end{cor}

\vskip 10pt
\begin{proof}
It suffices to show that if $g(t)=\sum_{k=0}^m a_kt^k$ is a polynomial of degree $m$ with
non-negative coefficients, then $g^{1/m}(t)$ is subadditive (on the positive half-line).
To this end, note that for $u\in(0,1)$, we have $g^{1/m}(ut) \ge ug^{1/m}(t)$ so that
$t\mapsto g^{1/m}(t)/t$ is non-increasing (this decreasing property of $w(t)/t$ is called
{\it quasi-concavity} for $w(t)$). Thus for all positive reals $a,\,b$,
$$
g^{1/m}(a) \ge \frac{a}{a+b} g^{1/m}(a+b)
\qquad{\mathrm{and}}\qquad g^{1/m}(b) \ge \frac{b}{a+b} g^{1/m}(a+b)
$$
so that
$$
g^{1/m}(a) +g^{1/m}(b) \ge g^{1/m}(a+b).
$$
\end{proof}

\vskip 10pt
The assumptions of  Corollary \ref{C-4.3} do not ensure that $A\mapsto \| g^{1/m}(A)\|$ is
subadditive on $\bM_n^+$. Indeed, for a non-negative function $h(t)$ vanishing at $0$,
the subadditivity of  $A\mapsto \| h(A)\|_1$, the trace norm, implies that $h(t)$ is
convex. A lot of polynomial $g(t)$ in Corollary \ref{C-4.3} may satisfy $g(0)=0$ and
$g^{1/m}(t)$ is not convex; for instance, if $m=3$ and $g(t)=t+t^3$.

Another example  satisfying the assumptions of Theorem \ref{T-4.1}\,(1), with $q=1/2$, is
$$
g(t)=t+(t-1)_+,
$$
where $(t-1)_+ := \max\{0,t-1\}$. Again, this follows from the fact that $\sqrt{g(t)}$ is
quasi-concave. 

Let $q\in(0,1)$. If $g_1,g_2:[0,\infty)\to [0,\infty)$ are convex functions such that
$g_1^q(t)$ and $g_2^q(t)$ are quasi-concave, then $g_1+g_2$ shares the same properties.
This provides a sub-class, closed under sum, of the class of functions satisfying the
assumptions of Theorem \ref{T-4.1}\,(1).

Thus, there are several examples of functions $g(t)$ for which Theorem \ref{T-4.1}\,(1) and
its corollaries below are available.

Now we turn to some applications of \eqref{F-4.1} to block-matrices.

\vskip 10pt 
\begin{cor}\label{C-4.4}
Let $g:[0,\infty)\to [0,\infty)$ be convex, $g(0)=0$, and let $ q\in(0,1)$.  
If $g^q(t)$ is subadditive, then
\begin{equation*}
\left\| \, g\left(\begin{bmatrix} A &X \\
X^* &B\end{bmatrix}\right) \right\| \, \le
\left( \left\|  g(A)\right\|^q + \left\|  g(B) \right\|^q \right)^{1/q}
\end{equation*}
for all partitioned positive operators and all symmetric norms.
\end{cor}

\vskip 5pt
\begin{proof}
The corollary is a straightforward consequence of the combination of \eqref{F-4.1} with the
decomposition \eqref{F-2.3} in $\bM_{n+m}^+$,
\begin{equation*}
\begin{bmatrix} A &X \\
X^* &B\end{bmatrix} = U
\begin{bmatrix} A &0 \\
0 &0\end{bmatrix} U^* +
V\begin{bmatrix} 0 &0 \\
0 &B\end{bmatrix} V^*
\end{equation*}
for some unitaries $U,\,V\in \bM_{n+m}$. To prove this, factorize positive matrices as a
square of positive matrices,
\begin{equation*}
\begin{bmatrix} A &X \\
X^* &B\end{bmatrix} =
\begin{bmatrix} C &Y \\
Y^* &D\end{bmatrix}
\begin{bmatrix} C &Y \\
Y^* &D\end{bmatrix}
\end{equation*}
and observe that it can be written as
\begin{equation*}
\begin{bmatrix} C &0 \\
Y^* &0\end{bmatrix}
\begin{bmatrix} C &Y \\
0 &0\end{bmatrix} +
\begin{bmatrix} 0 &Y \\
 0&D\end{bmatrix}
\begin{bmatrix} 0 &0 \\
Y^* &D\end{bmatrix} = T^*T + S^*S.
\end{equation*}
Then, use the fact that $T^*T$ and $S^*S$ are unitarily congruent to
$$
TT^*= \begin{bmatrix} A &0 \\
0 &0\end{bmatrix}
\quad \mathrm{and} 
\quad
SS^*=\begin{bmatrix} 0 &0 \\
0 &B\end{bmatrix}.
$$
\end{proof}

Of course, a version of this corollary holds for partitioned operators in $m^2$ blocks.
In particular, in case of the trace norm and an $m\times m$ matrix:

\vskip 10pt 
\begin{cor}\label{C-4.5}
Let $g:[0,\infty)\to [0,\infty)$ be convex, $g(0)=0$, and let $ q\in(0,1)$.  
If  $g^q(t)$  is subadditive, then
\begin{equation*}
{\mathrm{Tr\,}} g(A) \le \left( \sum_{i=1}^m g^q(a_{ii}) \right)^{1/q}
\end{equation*}
for all positive $m\times m$ matrix  $A=[a_{ij}]$.
\end{cor}

\vskip 10pt 
Corollary \ref{C-4.5} is unusual as it reverses the standard majorization inequality
$$
{\mathrm{Tr\,}} g(A) \ge  \sum_{i=1}^m g(a_{ii}).
$$

Corollary \ref{C-4.4} is understood with the natural convention that a symmetric norm
$\|\cdot\|$ on $\bM_{n+m}$ induces a symmetric norm on $\bM_n$, denoted with the same
symbol, by setting for all $A\in\bM_n$,
$$
 \| A\| :=\left\| \begin{bmatrix} A &0 \\
0 &0\end{bmatrix}\right\|.
$$
In the opposite way, starting from a symmetric norm $\|\cdot\|$  on $\bM_{n}$, we can
extend it to a symmetric norm  $\|\cdot\|_{\wedge}$ on $\bM_{n+m}$ by setting for all
$A\in\bM_{n+m}$,
$$
 \| A\|_{\wedge} := \| A^{\wedge} \|
$$
where $A^{\wedge}$ denotes the $n\times n$ diagonal matrix whose entries down to the
diagonal are the $n$ largest singular values of $A$. With this convention, we may remove
the assumption $g(0)=0$ in Corollary \ref{C-4.4}:

\vskip 10pt 
\begin{cor}\label{C-4.6}
Let $g:[0,\infty)\to [0,\infty)$ be convex and  $ q\in(0,1)$.  
If $g^q(t)$  is subadditive, then
\begin{equation*}
\left\| \, g\left(\begin{bmatrix} A &X \\
X^* &B\end{bmatrix}\right) \right\|_{\wedge} \, \le
\left( \left\|  g(A)\right\|^q + \left\|  g(B) \right\|^q \right)^{1/q}
\end{equation*}
for all positive operators  partitioned in blocks of same size and all symmetric norms.
\end{cor}

\vskip 10pt
The remaining part of this section deals with consequences of \eqref{F-4.2}.

\vskip 10pt
\begin{cor}\label{C-4.7}
Let $f(t)=a_1t +a_2t^{1/2}+\cdots +a_mt^{1/m}$ with all non-negative $a_k$'s. Then, for
all positive operators $A,\, B$ and all symmetric anti-norms,
\begin{equation*}
\| f(A+B) \|^m_! \ge \| f(A) \|^m_! +\| f(B) \|^m_!.
\end{equation*}
\end{cor}

\vskip 5pt
\begin{proof}
The proof is similar to that of Corollary \ref{C-4.3}: One first checks that
$t\mapsto f^m(t)/t$ is increasing on $(0,\infty)$ (this property is called
{\it quasi-convexity}).
\end{proof}

\vskip 10pt
Another example satisfying to the assumptions of \eqref{F-4.2}, with $p=2$, is
$$
f(t)=t-\frac{(t-1)_+}{2}.
$$
This follows from the fact that $f^2(t)$ is quasi-convex. 

Let $p\in(1,\infty)$. If $f_1, f_2:[0,\infty)\to [0,\infty)$ are concave
functions such that $f_1^p(t)$ and $f_2^p(t)$ are quasi-convex, then $f_1+f_2$ shares
the same properties. This provides a sub-class, closed under sum, of the class of
functions satisfying the assumptions of Theorem \ref{T-4.1}\,(2). Furthermore, under the
above assumptions, $\sqrt{f_1(t)f_2(t)}$ is also concave and  $\{\sqrt{f_1(t)f_2(t)}\}^p$
is quasi-convex. Thus, this sub-class of concave functions with quasi-convex $p$-powers is
invariant for both arithmetic and geometric means.

The next corollary is an extension on Minkowski's inequality (when $h(t)=t$).
 
\vskip 10pt
\begin{cor}\label{C-4.8}
Let $h:[0,\infty)\to [0,\infty)$ be superadditive. Assume that $h(t)$ is strictly positive
and $C^2$ on $(0,\infty)$ and that $(\log h(t))''<0$ for all $t>0$ (hence $h(t)$ is
strictly log-concave). Then, for all positive operators $A,\, B$ on an $n$-dimensional
space,
\begin{equation*}
{\det}^{1/n} h(A+B) \ge {\det}^{1/n} h(A) + {\det}^{1/n} h(B).
\end{equation*}
\end{cor}

\vskip 5pt
\begin{proof}
We may assume that $A,\, B$ and $A+B$ are invertible, thus with spectra lying on a compact
interval $[a,b]\subset(0,\infty)$. There exists a $q\in(0,1)$ small enough to ensure that 
$\exp\{q\log h(t)\}=h^q(t)$ is concave on $[a,b]$. Indeed, this is obvious since
$$
(h^q(t))''=qh^q(t)\left(q\{(\log h(t))'\}^2+(\log h(t))''\right).
$$
Note that to apply Theorem 4.1\,(2) to the operators $A,\, B$ and some function $f(t)$,
it suffices to have concavity of $f(t)$ on $[a,b]$ and superadditivity of $f^p(t)$ on
$[a,\infty)$. Thus, we may apply Theorem 4.1\,(2) to $f(t)=h^q(t)$ and
$p=1/q\in(1,\infty)$. This proves the corollary.
\end{proof}

\vskip 10pt 
\begin{cor}\label{C-4.9}
Let $f:[0,\infty)\to [0,\infty)$ be concave and  $ p\in(1,\infty)$. If  $f^p(t)$  is
superadditive, then, for all partitioned positive operators and all Schatten $r$-anti-norms,
\begin{equation*}
\left\| \, f\left(\begin{bmatrix} A &X \\
X^* &B\end{bmatrix}\right) \right\|_r \, \ge
\left( \left\|  f(A)\right\|_r^p + \left\|  f(B) \right\|_r^p\right)^{1/p}.
\end{equation*}
\end{cor}

This can be extended to all anti-norms $\|\cdot\|_!$ on $\bM_{n+m}$ with the convention
that it induces an anti-norm on $\bM_{n}$ by replacing $A\in \bM_{n}$ with
$A\oplus 0\in\bM_{n+m}$.

\vskip 10pt 
\begin{cor}\label{C-4.10}
Let $f:[0,\infty)\to [0,\infty)$ be concave and  $ p\in(1,\infty)$. If  $f^p(t)$  is
superadditive, then
\begin{equation*}
{\mathrm{Tr\,}} f(A) \ge \left( \sum_{i=1}^m f^p(a_{ii}) \right)^{1/p}
\end{equation*}
for all positive $m\times m$ matrix $A=[a_{ij}]$.
\end{cor}

\vskip 10pt 
Corollary \ref{C-4.10} is unusual as it reverses the standard majorization inequality
$$
{\mathrm{Tr\,}} f(A) \le  \sum_i f(a_{ii})
$$
that has been strengthened in Corollary \ref{C-2.5}.

\section{Convexity and concavity criteria for trace functionals}

Let $\tau:=(1/n){\mathrm{Tr}}$ denote the normalized trace on $\bM_n$.
In the preceding section we have studied functionals on $\bM_n^+$ of the type
$X\mapsto \| h(X)\|^r$ or $X\mapsto \| h(X)\|^r_!$
and obtained superadditivity or subadditivity results according to that $h^r(x)$ is
superadditive or subadditive. We may also address the question of convexity/concavity
properties of these functionals. In case of the trace norm, this leads us to consider conditions
on functions $\varphi(t)$ ensuring that $A\mapsto\varphi\circ\tau\circ f(A)$ is convex
(resp., concave) on $\bM_n^+$ whenever $\varphi\circ f(t)$ is convex (resp., concave) on
$[0,\infty)$. In this section we will treat the question in the setting of a general
interval $\Lambda\subset\bR$. Let $\bD_n\{\Xi\}$ denote the diagonal part of
$\bM_n\{\Xi\}$ for an interval $\Xi$.

\begin{prop}\label{P-5.1}
Let $\varphi(t)$ be a strictly increasing continuous function on an interval $\Lambda$
and $\Xi:=\varphi(\Lambda)$. If the functional
$A\mapsto\varphi\circ\tau\circ\varphi^{-1}(A)$ is convex (resp., concave) on
$\bM_n\{\Xi\}$, or on  $\bD_n\{\Xi\}$, then $\varphi$ is necessarily concave (resp.,
convex) on $\Lambda$.
\end{prop}

\begin{proof}
 The assumption of convexity (resp., concavity) on $\bD_n\{\Xi\}$ means the following condition:
\begin{itemize}
\item[(c$'$)] The $n$-variable function
$$
(x_1,\dots,x_n)\in\Xi^n\mapsto\varphi\Biggl(
\frac{1}{n}\sum_{i=1}^n\varphi^{-1}(x_i)\Biggr)
$$
is convex (resp., concave).
\end{itemize}

Thus, we may show that convexity condition (c$'$) implies the concavity of $\varphi$ (the
other case is similar). For $\vec x=(x_1,\dots,x_n)\in\Xi^n$ we use the notations
$\tau(\vec x):=n^{-1}\sum_{i=1}^nx_i$ and
$\varphi^{-1}(\vec x):=(\varphi^{-1}(x_1),\dots,\varphi^{-1}(x_n))$. Consider the cyclic
permutations $\vec x^{(0)}:=\vec x$ and $\vec x^{(k)}:=(x_{k+1},\dots,x_n,x_1,\dots,x_k)$
for $k=1,2,\dots,n-1$. Since
$n^{-1}\sum_{k=0}^n\vec x^{(k)}=(\tau(\vec x),\dots,\tau(\vec x))$, we have
$$
\tau(\vec x)=\varphi\circ\tau\circ\varphi^{-1}\Biggl(
\frac{1}{n}\sum_{k=0}^{n-1}\vec x^{(k)}\Biggr)
\le\frac{1}{n}\sum_{k=0}^{n-1}\varphi\circ\tau\circ\varphi^{-1}(\vec x^{(k)})
=\varphi\circ\tau\circ\varphi^{-1}(\vec x).
$$
Since $\varphi^{-1}$ is increasing on $\Xi$, we have
$\varphi^{-1}\circ\tau(\vec x)\le\tau\circ\varphi^{-1}(\vec x)$, which means that
$\varphi^{-1}$ is convex on $\Xi$, hence $\varphi$ is concave on $\Lambda$.
\end{proof}

The above proposition shows that the concavity or convexity of $\varphi(t)$ is necessary
to state our criteria in the next theorem. Thus, there is no serious loss of generality
by assuming that $\varphi(t)$ is smooth with the strictly negative or positive second
derivative. This  assumption yields a simple analytic criterion.

\begin{theorem}\label{T-5.2}
Let $\varphi(t)$ be a continuous function on an interval $\Lambda$ and
$\Xi:=\varphi(\Lambda)$. Assume that $\varphi(t)$ is $C^2$ on $\Lambda^\circ$, the interior
of $\Lambda$, and that $\varphi'(t)>0$ and $\varphi''(t)<0$ (resp., $\varphi''(t)>0$) on
$\Lambda^\circ$. Then the following conditions are equivalent:
\begin{itemize}
\item[(a)] The function $\varphi'(t)/\varphi''(t)$ is convex (resp., concave) on
$\Lambda^\circ$.
\item[(b)] The functional $A\mapsto\varphi\circ\tau\circ\varphi^{-1}(A)$ is convex
(resp., concave) on $\bM_n\{\Xi\}$.
\item[(c)] The functional $A\mapsto\varphi\circ\tau\circ\varphi^{-1}(A)$ is convex
(resp., concave) on $\bD_n\{\Xi\}$.
\item[(d)] For  any interval $\Omega$ and any function $f:\Omega\to\Lambda$ such that
$\varphi\circ f(t)$ is convex (resp, concave), the functional
$A\mapsto\varphi\circ\tau\circ f(A)$ is convex (resp., concave) on $\bM_n\{\Omega\}$.
\end{itemize}
\end{theorem}

\begin{proof}
We prove the convexity case in (a)--(d). First, assume only that $\varphi(t)$ is a
strictly increasing continuous function on $\Lambda$, and we show that each of (b), (c),
and (d) are equivalent to condition (c$'$) in the proof of Proposition \ref{P-5.1}.
It is obvious that (d) $\Rightarrow$ (b) $\Rightarrow$ (c) $\Leftrightarrow$ (c$'$).
To prove (c$'$) $\Rightarrow$ (d), let $f$ be as stated
in (d). Then $f=\varphi^{-1}\circ\varphi\circ f$ is automatically convex on $\Omega$
since $\varphi^{-1}$ is increasing and convex on $\Xi$ as shown in the proof of Proposition
\ref{P-5.1}. Let $A,B\in\bM_n\{\Omega\}$. By the Ky Fan majorization
$(A+B)/2\prec(A^\downarrow+B^\downarrow)/2$ in the same notation as in the proof of
Theorem \ref{T-4.1}, we have
\begin{align*}
f\biggl(\frac{A+B}{2}\biggr)
&\prec_wf\biggl(\frac{A^\downarrow+B^\downarrow}{2}\biggr)
=\varphi^{-1}\circ\varphi\circ f\biggl(\frac{A^\downarrow+B^\downarrow}{2}\biggr) \\
&\le\varphi^{-1}\biggl(\frac{\varphi\circ f(A^\downarrow)
+\varphi\circ f(B^\downarrow)}{2}\biggr)
\end{align*}
so that, by using (c$'$),
\begin{align*}
\varphi\circ\tau\circ f\biggl(\frac{A+B}{2}\biggr)
&\le\varphi\circ\tau\circ\varphi^{-1}
\biggl(\frac{\varphi\circ f(A^\downarrow)+\varphi\circ f(B^\downarrow)}{2}\biggr) \\
&\le\frac{\varphi\circ\tau\circ f(A^\downarrow)
+\varphi\circ\tau\circ f(B^\downarrow)}{2} \\
&=\frac{\varphi\circ\tau\circ f(A)+\varphi\circ\tau\circ f(B)}{2}.
\end{align*}
Hence it has been proved that (b)--(d) and (c$'$) are equivalent.

Next, assume that $\varphi(t)$ is smooth as stated in the theorem with $\varphi''(t)<0$,
and we prove the equivalence between  (c$'$) and (a). Let $t_i\in\Lambda^\circ$ and
$s_i:=\varphi(t_i)\in\Xi^\circ$ for $1\le i\le n$. For every $x_i\in\bR$, $1\le i\le n$,
one can directly compute the second derivative
\begin{align*}
&\frac{d^2}{du^2}\,\varphi\Biggl(\frac{1}{n}\sum_{i=1}^n\phi^{-1}(s_i+ux_i)
\Biggr)\Bigg|_{u=0} \\
&\qquad=\varphi''\Biggl(\frac{1}{n}\sum_{i=1}^nt_i\Biggr)
\Biggl(\frac{1}{n}\sum_{i=1}^n\frac{x_i}{\varphi'(t_i)}\Biggr)^2
-\varphi'\Biggl(\frac{1}{n}\sum_{i=1}^nt_i\Biggr)
\Biggl(\frac{1}{n}\sum_{i=1}^n\frac{\varphi''(t_i)x_i^2}{\varphi'(t_i)^3}\Biggr).
\end{align*}
Hence condition (c$'$) is satisfied if and only if
\begin{equation}\label{F-5.1}
\Biggl(\frac{1}{n}\sum_{i=1}^n\frac{x_i}{\varphi'(t_i)}\Biggr)^2
\le\frac{\varphi'\bigl(\frac{1}{n}\sum_{i=1}^nt_i\bigr)}
{\varphi''\bigl(\frac{1}{n}\sum_{i=1}^nt_i\bigr)}
\Biggl(\frac{1}{n}\sum_{i=1}^n\frac{\varphi''(t_i)x_i^2}{\varphi'(t_i)^3}\Biggr)
\end{equation}
holds for all $t_i\in\Lambda^\circ$ and $x_i\in\bR$. If this holds, then letting
$x_i:=\varphi'(t_i)^2/\varphi''(t_i)$ gives, thanks to $\varphi'(t)/\varphi''(t)<0$,
$$
\frac{1}{n}\sum_{i=1}^n\frac{\varphi'(t_i)}{\varphi''(t_i)}
\ge\frac{\varphi'\bigl(\frac{1}{n}\sum_{i=1}^nt_i\bigr)}
{\varphi''\bigl(\frac{1}{n}\sum_{i=1}^nt_i\bigr)},
$$
which means that $\varphi'(t)/\varphi''(t)$ is convex on $\Lambda^\circ$. Conversely, if
$\varphi'(t)/\varphi''(t)$ is convex on $\Lambda^\circ$, then we have \eqref{F-5.1} as
follows:
\begin{align*}
\Biggl(\frac{1}{n}\sum_{i=1}^n\frac{x_i}{\varphi'(t_i)}\Biggr)^2
&=\Biggl(\frac{1}{n}\sum_{i=1}^n\frac{\varphi'(t_i)^{1/2}}{\{-\varphi''(t_i)\}^{1/2}}
\cdot\frac{\{-\varphi''(t_i)\}^{1/2}x_i}{\varphi'(t_i)^{3/2}}\Biggr)^2 \\
&\le\frac{\varphi'\bigl(\frac{1}{n}\sum_{i=1}^nt_i\bigr)}
{\varphi''\bigl(\frac{1}{n}\sum_{i=1}^nt_i\bigr)}
\Biggl(\frac{1}{n}\sum_{i=1}^n\frac{\varphi''(t_i)x_i^2}{\varphi'(t_i)^3}\Biggr)
\end{align*}
by using the Schwarz inequality. Thus (c$'$) follows.

The concavity case is similarly proved, or else one can reduce the assertion to the
convexity case by considering $-\varphi(-t)$ on $-\Lambda$ and $-f(t)$ on $\Omega$.
\end{proof}

For given functions $\varphi$ on $\Lambda$ and $f:\Omega\to\Lambda$, if we define
$\tilde\varphi(t):=\varphi(-t)$ on $-\Lambda$ and $\tilde f(t):=-f(t)$ on $-\Lambda$, then
$\tilde\varphi\circ\tilde f=\varphi\circ f$ and
$\tilde\varphi\circ\tau\circ\tilde f=\varphi\circ\tau\circ f$ on $\Omega$. By taking
account of these, we see that Theorem \ref{T-5.2} holds true also when the assumption
$\varphi'(t)>0$ is replaced with $\varphi'(t)<0$.

The following examples are applications of Theorem \ref{T-5.2}:

\begin{example}\label{E-5.3}
When $\varphi(t)=\log t$ on $(0,\infty)$, $\varphi'(t)>0$, $\varphi''(t)<0$ and
$\varphi'(t)/\varphi''(t)=-t$. Hence, if $f:\Omega\to\bR$ is convex, then
$\log\tau\bigl(e^{f(A)}\bigr)$ (also $\log\mathrm{Tr}\,e^{f(A)}$) is convex on
$\bM_n\{\Omega\}$. In particular, $\log\mathrm{Tr}\,e^A$ is called the pressure of $A=A^*$
and its convexity in $A$ is well-known.
\end{example}

\begin{example}\label{E-5.4}
In case of $\varphi(t)=e^t$ on $\bR$, $\varphi'(t)/\varphi''(t)=1$. Hence, if
$f:\Omega\to(0,\infty)$ is concave, then $\det^{1/n}f(A)=\exp\tau(\log f(A))$ is concave
on $\bM_n\{\Omega\}$. By continuity this is true if $f$ is non-negative and concave on
$\Omega$; hence \eqref{F-2.7} is recaptured.
\end{example}

\begin{example}\label{E-5.5}
When $\varphi(t)=t^r$ on $[0,\infty)$ with $r\in(0,\infty)\setminus\{1\}$,
$\varphi'(t)/\varphi''(t)=(r-1)^{-1}t$. Hence,
$\|f(A)\|_{1/r}=\{\mathrm{Tr}\,f(A)^{1/r}\}^r$ is convex (resp., concave) on
$\bM_n\{\Omega\}$ if $r\in(0,1)$ (resp., $r\in(1,\infty)$) and $f:\Omega\to[0,\infty)$ is
convex (resp., concave). We thus have the convexity of Schatten norms and the concavity of
Schatten anti-norms involving a convex/concave function $f(t)$. A stronger statement has
been obtained in Corollary \ref{C-3.12}.
\end{example}

Finally, we state an abstract version of Theorem \ref{T-4.1}. In fact,  in case of
$\varphi(t)=t^r$, combined with Proposition \ref{P-3.7} and its norm version, it becomes
the  superadditivity part (when $r\in(1,\infty)$) of Theorem \ref{T-4.1} and its
subadditivity version (when $r\in(0,1)$). The proof is essentially the same as that of
Theorem \ref{T-4.1}.

\begin{prop}\label{P-5.6}
Let $\varphi$ be a strictly increasing continuous function from $[0,\infty)$ onto itself.
Let $\|\cdot\|$ (resp., $\|\cdot\|_!$) be a symmetric norm (resp., anti-norm) and $\Phi$
(resp., $\Phi_!$) be the corresponding gauge (resp., anti-gauge, see Proposition
\ref{P-3.2}) function on $\bR_n^+$. Assume that the $n$-variable function
$\varphi\circ\Phi\circ\varphi^{-1}(\vec x)$ (resp.,
$\varphi\circ\Phi_!\circ\varphi^{-1}(\vec x)$) is convex (resp., concave) on $\bR_+^n$.
Then for any convex (resp., concave) function $f:[0,\infty)\to[0,\infty)$ such that
$\varphi\circ f$ is subadditive (resp., superadditive), the functional
$A\mapsto\varphi(\|f(A)\|)$ (resp., $\varphi(\|f(A)\|_!)$) is subadditive (resp.,
superadditive) on $\bM_n^+$.
\end{prop}


\vskip 20pt\noindent

\end{document}